\documentclass{amsart}
\usepackage{subcaption}
\usepackage{amsrefs}
\usepackage{blkarray}
\usepackage{parskip}
\usepackage{xcolor}
\usepackage{stmaryrd}
\usepackage{graphicx}
\usepackage[export]{adjustbox}

\theoremstyle{plain}
\newtheorem{thm}{Theorem}
\newtheorem{prop}[thm]{Proposition}
\newtheorem{lemma}[thm]{Lemma}

\theoremstyle{definition}
\newtheorem{dfn}[thm]{Definition}
\newtheorem{exam}[thm]{Example}

\theoremstyle{remark}

\newtheorem*{rem*}{Remark}
\numberwithin{equation}{section}

\newcommand{\C}{\mathbb C}
\newcommand{\R}{\mathbb R}
\newcommand{\GL}{{\rm GL}}

\DeclareMathOperator{\Sp}{Sp}
\DeclareMathOperator{\M}{M}

\newcommand{\gl}{\mathfrak{gl}}
\newcommand{\g}{\mathfrak g}
\renewcommand{\k}{\mathfrak k}

\newcommand{\LRC}[2]{{\rm LRC}\:^{#1}_{#2}}
\newcommand{\la}{\lambda}
\newcommand{\al}{\alpha}
\newcommand{\be}{\beta}

\renewcommand{\b}{\mathbf b}
\renewcommand{\c}{{\rm c}}
\DeclareMathOperator{\nc}{nc}

\DeclareMathOperator{\ch}{ch}
\DeclareMathOperator{\U}{U}

\DeclareMathOperator{\SU}{SU}
\renewcommand{\S}{{\rm S}}
\DeclareMathOperator{\SL}{SL}
\DeclareMathOperator{\Sym}{Sym}
\renewcommand{\sl}{\mathfrak{sl}}
\newcommand{\s}{\mathfrak s}
\renewcommand{\u}{\mathfrak u}
\newcommand{\su}{\mathfrak{su}}
\newcommand{\D}{\mathbb D}
\renewcommand{\u}{\mathfrak u}
\newcommand{\m}{\mathfrak m}

\begin{document}

\title[Enright resolutions and Blattner's formula]{Enright resolutions encoded by a generating function for Blattner's formula: Type A}

\author{William Q. Erickson}
\address{
William Q. Erickson\\
University of Wisconsin--Milwaukee \\
3200 North Cramer Street\\
Milwaukee, WI 53211}
\email{wqe@uwm.edu}

\begin{abstract}
Consider the classical action of $\GL_n$ on a sum of $q$ copies of the defining representation and $p$ copies of its dual; by Howe duality, the polynomial functions on this space decompose under the joint action of $\GL_n$ and $\gl_{p+q}$.  The modules for $\gl_{p+q}$ are infinite-dimensional and their structure is complicated outside a certain stable range, although Enright and Willenbring (2005) constructed resolutions in terms of generalized Verma modules.  We show that these resolutions can be read off from the coefficients in a formal series arising in an entirely different setting --- discrete series representations of $\SU(n,p+q)$ in the case of two noncompact simple roots.
\end{abstract}

\subjclass[2010]{Primary 20G20; Secondary 17B10; 05E10}

\keywords{Classical invariant theory, Howe duality, generalized Verma modules, Enright resolutions, discrete series, Blattner's formula}

\maketitle


\section{Introduction}
\label{s:intro}

The result in this paper is a formal series which can be interpreted from two seemingly unrelated perspectives: the classical invariant theory of $\GL_n$, and $\k$-type multiplicities in discrete series representations of $\SU(n,m)$.  

Our first setting is the classical action of the complex general linear group $\GL_n$ on $V=\M_{n,p} \oplus \M_{n,q}$.  As a result of Howe duality, the polynomial functions on $V$ decompose under the joint action of $\GL_n \times \gl_{p+q}$:
$$
\C[V] \cong \bigoplus_\la F^\la_n \otimes \widetilde F^\la_{p,q},
$$
where the $F^\la_n$ are rational $\GL_n$-representations, and the $\widetilde F^\la_{p,q}$ are infinite-dimensional $\gl_{p+q}$-modules.  Although the modules $\widetilde F^\la_{p,q}$ have a complicated structure (outside the stable range $n \geq p+q$), they nonetheless have finite resolutions (constructed by Enright and Willenbring in \cite{ew}) in terms of generalized Verma modules.  It is these Enright resolutions which we seek to organize in one place.  

Our second setting is, it seems, far removed from the first: namely, the special indefinite unitary group $\SU(n,p+q)$ and certain of its discrete series representations in the case of two noncompact simple roots. Harish--Chandra first conjectured (1954), and subsequently proved in \cite{hc}, that a linear connected semisimple real Lie group $G_0$ has discrete series representations if and only if its rank is equal to that of a maximal compact subgroup $K_0$.  It is then natural to restrict a discrete series representation of $G_0$ to the action of the complexified Lie algebra $\k$, and to seek the multiplicities of irreducible $\k$-modules (``$\k$-types'') in the resulting decomposition.  A formula for these $\k$-type multiplicities --- an alternating sum over the Weyl group of $\k$ --- was (according to Harish--Chandra) conjectured by Robert J.\ Blattner, and this ``Blattner's formula'' was eventually proved by Hecht and Schmid in \cite{hs}.   (We note that outside certain conditions, the outputs of Blattner's formula do not correspond to actual $\k$-type multiplicities, and can assume negative values.)  More recently, in \cite{wz}, Willenbring and Zuckerman derived an elegant generating function for a formal series $\b(0)$, which encodes the values of Blattner's formula for the trivial $\k$-type.   This series $\b(0)$ is our focus in the present paper.

Theoretically, the heart of the paper lies in Lemmas \ref{lemma:deltabarcancel} and \ref{lemma:V_iso_u'+}, since these spell out the explicit connection between the character theory of $\C[V]$, on one hand, and the series $\b(0)$ for $\SU(n,p+q)$ on the other hand.  Perhaps more interesting than the main result itself (Theorem \ref{thm:mainresult}), however, is its application in unraveling the Enright resolutions.  Specifically, we show (see Example \ref{ex:n1p3q3}) how to write down the Enright resolution of $\widetilde F^\la_{p,q}$ --- without any knowledge of the machinery used in \cite{ew} to construct the resolutions --- simply by reading off the  coefficient of $e^{\la}$ in $\b(0)$.

\section{Classical invariant theory}
\label{s:classicalinvariant}

We begin by recalling a result of classical invariant theory and Howe duality in Type A.  We largely follow the exposition in Section 3 of \cite{htw}.  Fix $n,p,q \in \mathbb N$, and let
$$
V = \M_{n,p}\oplus \M_{n,q}
$$
where the matrices are over $\C$.  Let $\C[V]$ denote the space of polynomial functions on $V$.  Now consider the pair $(\U(n), \U(p,q))$.  This is known as a \textit{compact dual pair}, since the first element is a compact group, and since the two groups centralize each other as subgroups of $\Sp(2n(p+q),\R)$.  In order to study the action of the dual pair $(\U(n), \U(p,q))$ on $\C[V]$, it will be convenient to consider instead the complexification of the first group, namely $\GL_n=\GL(n,\C)$, and the complexified Lie algebra of the second group, namely $\gl_{p+q} = \gl(p+q,\C)$.  This will allow us to avoid technicalities involving a covering group of $\U(p,q)$.  Hence, we consider the pair $(\GL_n,\gl_{p+q})$ and its action on $\C[V]$.

We let $\GL_n$ act by matrix multiplication on the second summand $\M_{n,q}$ of $V$, and by the inverse transpose (denoted by the superscript $-T$) on the first summand $\M_{n,p}$:
\begin{equation}
\label{GLnaction}
g \cdot (X,Y) = (g^{-T}X, gY).
\end{equation}
In classical invariant theory, this is the natural action obtained by regarding $\M_{n,q}$ as the sum of $q$ copies of the defining representation $\C^n$, and $\M_{n,p}$ as the sum of $p$ copies of the dual representation $(\C^n)^*$.  The action \eqref{GLnaction} induces the usual action on $\C[V]$ via
\begin{equation*}
g \cdot f(X,Y) = f(g^T X, g^{-1}Y).
\end{equation*}

This classical action of $\GL_n$ on $\C[V]$ extends to the Weyl algebra $\D(V)$ of polynomial-coefficient differential operators on $\C[V]$; in particular, $\GL_n$ normalizes $\D(V)$ inside ${\rm End}(\C[V])$.  Let $\D(V)^{\GL_n}$ be the $\GL_n$-invariant subalgebra, i.e., the subalgebra containing those operators that commute with the $\GL_n$-action on $\C[V]$.  Then $\D(V)$ is generated (as an algebra) by a subset which is isomorphic to $\gl_{p+q}$ as a Lie algebra.  (See Section 2 of \cite{cew} for the explicit isomorphism.)  This fact implies the following Howe duality decomposition as a module for $\GL_n \times \gl_{p+q}$:
\begin{equation}
\label{howedecomp}
\C[V] \cong \bigoplus_{\la} F_n^\la \otimes \widetilde F^\la_{p,q},
\end{equation}
where
\begin{itemize}
\item $\la$ ranges over a certain subset of the highest weights indexing the rational representations of $\GL_n$;
\item $F^\la_n$ is the rational representation of $\GL_n$ with highest weight $\la$;
\item $\widetilde F^\la_{p,q}$ is an infinite-dimensional $\gl_{p+q}$-module uniquely determined by $\la$.
\end{itemize}
For the purposes of this paper, we will actually restrict $\U(p,q)$ to the subgroup $G_0 = \SU(p,q)$ because the latter is a Hermitian symmetric group.  Therefore, we will henceforth write $\g_0 = \su(p,q)$ and $\g=(\g_0)^\C = \sl_{p+q} \subset \gl_{p+q}$.

Consider the subalgebra $\k_0 = \s(\u(p)\oplus \u(q)) \subset \g_0$, embedded block-diagonally.  This $\k_0$ is the Lie algebra of the maximal compact subgroup $K_0 = \S(\U(p) \times \U(q)) \subset G_0$, for which reason we say that $\k_0$ is a ``maximal compact subalgebra'' of $\g_0$. (The symbols ``$\S$'' and ``$\s$'' denote the determinant-1 and trace-0 conditions, respectively.)  The complexification $\k = (\k_0)^\C = \s(\gl_p \oplus \gl_q) \subset \g$ acts on $V$ by the derived action (up to a central character) of the complexified group $K = \S(\GL_p \times \GL_q)$:
\begin{equation}
    \label{Kaction}
    (g,h) \cdot (X,Y) = (Xg^{-1},Yh^{-1}),
\end{equation}
where $g \in \GL_p$ and $h \in \GL_q$. Since the action in \eqref{Kaction} commutes with the $\GL_n$-action in \eqref{GLnaction}, we have an action of the product $M'=\S(\GL_n \times \GL_p \times \GL_q)$ upon $V$:
\begin{equation}
    \label{Maction}
    (g_n,g_p,g_q) \cdot (X,Y) = (g_n^{-T}Xg_p^{-1},g_nYg_q^{-1}).
\end{equation}
(We choose the notation $M'$ because this group will soon play the role of a Levi factor in a different setting --- a setting in which we will decorate all groups and Lie algebras with prime symbols in order to distinguish them from the present context.)

We now clarify which $\la$'s appear in the decomposition \eqref{howedecomp}.  Recall that rational representations $F^\la_n$ of $\GL_n$ are indexed by their highest weights, which (in standard coordinates) range over all weakly decreasing integer $n$-tuples $\la=(\la_1,\ldots,\la_n)$.  If the $\la_i$ are all nonnegative, then $\la$ is called a \textbf{partition}, and $F_n^\la$ is a polynomial representation of $\GL_n$.  (``Rational'' and ``polynomial'' refer to the matrix entry functions in the image of the representation, in terms of the coordinate functions on $\GL_n$.)  Note that any weakly decreasing $n$-tuple $\la$ can be written as a pair $[\la^+,\la^-]$ of partitions $\la^+$ and $\la^-$, of lengths $a$ and $b$ respectively (where $a+b\leq n$), by setting
$$
\la = [\la^+,\la^-] :=  (\la^+_1,\ldots,\la^+_a,0,\ldots,0,-\la^-_b,\ldots,-\la^-_1).
$$
For example, if $\la=(5,2,0,0,-1,-3,-4)$, then $\la^+=(5,2)$ and $\la^- = (4,3,1)$.  We will consider two partitions equivalent if they differ only by padding zeros at the end; thus $(5,2)$ is the same as $(5,2,0,0)$.  Let $\ell(-)$ denote the length of a partition, disregarding trailing zeros.  

The sum in \eqref{howedecomp} ranges over all weakly decreasing $n$-tuples $\la$ such that $\ell(\la^+)\leq p$ and $\ell(\la^-)\leq q$.  These length conditions motivate a notion of ``stability'': we say that the parameters $n,p,q$ lie in the \textbf{stable range} when $$n \geq p+q.$$  Inside the stable range, any pair of partitions $(\la^+,\la^-)$ such that $\ell(\la^+)\leq p$ and $\ell(\la^-)\leq q$ determines a weakly decreasing $n$-tuple $\la = [\la^+,\la^-]$, i.e., the highest weight of the rational $\GL_n$-representation $F^\la_n$.  Furthermore, inside the stable range, upon restriction to $K = \S(\GL_p \times \GL_q)$, the module $\widetilde F^\la_{p,q}$ has the especially nice structure
\begin{equation}
\label{kdecomp}
     \widetilde F^\la_{p,q} \cong \C[\M_{p,q}] \otimes (F^{\la^+}_p \otimes F^{\lambda^-}_q),
\end{equation}
where $K$ acts on the second tensor factor in the obvious way, and on the first tensor factor via
\begin{equation}
\label{kactiononMpq}
(g,h)\cdot f(X) = f(g^T X h).
\end{equation}  The reader may recognize \eqref{kdecomp} as the generalized Verma module for $\g$ induced from the irreducible $\k$-module $F^{\la^+}_p \otimes F^{\lambda^-}_q$.  In this context, $\C[\M_{p,q}] \cong \Sym(\mathfrak p^-)$ as $\k$-modules, given the Cartan decomposition $\g = \mathfrak p^- \oplus \k \oplus \mathfrak p^+$.  (We will see in the following section that inside the stable range, $\widetilde F^\la_{p,q}$ is a discrete series representation of $G_0$ as well.)

\subsection*{Outside the stable range: Enright resolutions}

Outside the stable range, the structure of $\widetilde F^\la_{p,q}$ is more complicated than in \eqref{kdecomp}.  Regardless of stability, however, $\widetilde F^\la_{p,q}$ admits a certain finite resolution, as proved by Enright and Willenbring in \cite{ew}.  This idea is an extension of the BGG resolution of a finite-dimensional module in terms of Verma modules, and of the Lepowsky resolution of a finite-dimensional module in terms of generalized Verma modules.   

In particular, there exists an \textbf{Enright resolution} (also called a ``generalized BGG resolution'' in \cite{eh}) for any unitarizable highest-weight representation of a simple connected real Lie group ($G_0 = \SU(p,q)$ in our case), given a maximal compactly embedded subgroup ($K_0 = \S(\U(p)\times \U(q))$ for us) such that $(G_0,K_0)$ is a Hermitian symmetric pair.  (As we will see later, the block-diagonal embedding $K_0 \subset G_0$ implies the Hermitian symmetric condition.)  Each of the finitely many terms in the Enright resolution is the direct sum of generalized Verma modules
\begin{equation}
\label{gvm}
U(\g) \otimes_{U(\mathfrak q)} L_\k(\xi)
\end{equation}
associated with the standard maximal parabolic subalgebra $\mathfrak q = \k \oplus \mathfrak p^+$, where $L_\k(\xi)$ denotes the irreducible $\k$-module with highest weight $\xi$, regarded as a $\mathfrak q$-module by letting $\mathfrak p^+$ act trivially.  As a $\k$-module, \eqref{gvm} is isomorphic to $\Sym(\mathfrak p^-) \otimes L_\k(\xi)$.  Therefore in our case, namely the Enright resolution of $\widetilde F^\la_{p,q}$, the generalized Verma modules restrict as $\k$-modules to the form 
\begin{equation}
    \label{gvmform}
M_{\mu,\nu}:=\C[\M_{p,q}] \otimes (F^{\mu}_p \otimes F^\nu_q),
\end{equation} where $\mu$ and $\nu$ are partitions of lengths at most $p$ and $q$ respectively.  See Theorem 2 and the preceding discussion in \cite{ew} for the explicit construction of the individual terms in the resolution; for our purposes, we will appeal only to its existence, because our main result will actually show us how to find all the data from the Enright resolution hidden within a certain generating function.  (See Example \ref{ex:n1p3q3}.) 

It will be useful to assign to each generalized Verma module $M_{\mu,\nu}$ an integer $\varepsilon^\la_{\mu,\nu}$ which describes its occurrence in the Enright resolution of $\widetilde F^\la_{p,q}$; morally, we want
$$
\varepsilon^\la_{\mu,\nu} = \begin{cases}
+1, & \text{$M_{\mu,\nu}$ is a summand in an even term in the resolution of $\widetilde F^\la_{p,q}$},\\
-1, & \text{$M_{\mu,\nu}$ is a summand in an odd term in the resolution of $\widetilde F^\la_{p,q}$},\\
0, & M_{\mu,\nu} \text{ does not appear anywhere in the resolution of }\widetilde F^\la_{p,q}.
\end{cases}
$$
(In theory we might also have values besides $\pm 1$ and $0$, if several copies of $M_{\mu,\nu}$ appear in a resolution.)  To formalize this notion, let $\ch_\k(-)$ denote the $\k$-character of a $\k$-module.  Let $$
\mathcal X:=\{ \ch_\k \left( M_{\mu,\nu} \right) \mid \ell(\mu)\leq p \text{ and }\ell(\nu)\leq q\}
$$
and consider its integer span $\mathbb Z\mathcal X$. (We can think of this as the Grothendieck group of the generalized Verma modules.) Then we can naturally regard the Euler characteristic of an Enright resolution as an element of $\mathbb Z\mathcal X$. In fact $\mathcal X$ is a basis for $\mathbb Z\mathcal X$, and so $\ch_\k ( \widetilde F^\la_{p,q})$ has a unique expansion as an element of $\mathbb Z\mathcal X$.  

\begin{dfn}
\label{def:epsilon}
For fixed $n,p,q$, suppose $\la,\mu,\nu$ satisfy the following conditions:
\begin{equation}
    \label{howeconditions}
    \la \text{ a weakly decreasing $n$-tuple}; \quad \ell(\la^+), \ell(\mu) \leq p; \quad \ell(\la^-),\ell(\nu)\leq q.
\end{equation}
Then we define $$\varepsilon^\la_{\mu,\nu}:= \text{the coefficient of $\ch_\k \left(M_{\mu,\nu}\right)$ in }\ch_\k(\widetilde F^\la_{p,q}),$$
where $\ch_\k(\widetilde F^\la_{p,q})$ is expanded as an element of $\mathbb Z \mathcal X$.
\end{dfn}


Whether or not we are in the stable range, we thus have a $\k$-character of the form
\begin{equation}
\label{ftildechar}
    \ch_\k (\widetilde F^\la_{p,q}) = \ch_\k \left(\C[\M_{p,q}]\right) \cdot \sum_{\mu,\nu} \varepsilon^\la_{\mu,\nu} \cdot \ch_\k( F^{\mu}_p \otimes F^\nu_q),
\end{equation}
where $\varepsilon^\la_{\mu,\nu} \in \mathbb Z$ is nonzero for only finitely many $\mu,\nu$.  Note that inside the stable range, \eqref{kdecomp} implies that \begin{equation} 
\label{epsilonstable}
    \varepsilon^\la_{\mu,\nu}=\delta_{(\la^+,\la^-),\:(\mu,\nu)}
\end{equation}
where $\delta$ is the Kronecker delta.

We can now write out the character of $\C[V]$ as a representation of $M'$, under the action in \eqref{Maction}.  Combining \eqref{howedecomp} and \eqref{ftildechar}, we have the factorization
\begin{equation}
\label{bigcharacter}
    \ch \C[V] = \ch \C[\M_{p,q}] \cdot \sum_{\la,\mu,\nu} \varepsilon^\la_{\mu,\nu} \cdot \ch (F_n^\la \otimes F^{\mu}_p  \otimes F^\nu_q)
\end{equation}
as a character of $M'$ and therefore of $\m'$.  This time the sum is infinite, since $\la$ ranges over an infinite set.

\section{Discrete series and Blattner's formula}
\label{s:Blattner}

In this section, we temporarily forget the concrete setting of the previous section, and so we regard $G_0$, $K_0$, $\g$, $\k$, etc.\ in the abstract.  Nonetheless, the reader may profitably keep in mind the specific notation of Section \ref{s:classicalinvariant}, since we will see that it furnishes an example of the present perspective.

A \textbf{discrete series representation} of a Lie group $G_0$ is a topologically closed subspace of $L^2(G_0)$ that is invariant and irreducible under the usual $G_0$-action on functions.  Harish--Chandra proved in \cite{hc} that a connected semisimple group $G_0$ has discrete series representations if and only if $\text{rank }G_0 = \text{rank }K_0$, where $K_0 \subset G_0$ is a maximal compact subgroup.  (Hence $G_0$ must have finite center.)

It is natural to restrict a discrete series representation from $G_0$ to $K_0$, thus to the complexified Lie algebra $\k$, then decompose into a direct sum of irreducible finite-dimensional $\k$-representations (``$\k$-types''), and then seek the multiplicity of a given $\k$-type in this decomposition.  Before recording this multiplicity formula (named after Robert J. Blattner), we explain the preliminary notation, following that used by Willenbring and Zuckerman in \cite{wz}.

Let $G_0$ be a connected, semisimple Lie group with finite center, with complexified Lie algebra $\g$.  Let $K_0 \subset G_0$ be a maximal compact subgroup, with complexified Lie algebra $\k$.  Fix a Cartan subalgebra $\mathfrak h \subset \g$.  By Harish--Chandra's equal-rank criterion, we may assume that $\mathfrak h \subset \k$.  Let $\Phi = \Phi(\g,\mathfrak h)$ denote the corresponding root system.

Let $\Phi^+$ denote a choice of positive roots and let let $\Phi^- = -\Phi^+$ denote the set of negative roots.  We call a root $\al \in \Phi$ a ``compact root'' if its corresponding root space $\g_\al \subset \k$; otherwise, we call the root ``noncompact.''  We denote the sets of compact and noncompact roots by $\Phi_{\c}$ and $\Phi_{\nc}$ respectively.  Thus we have $\k = \mathfrak h \oplus \bigoplus_{\al \in \Phi_{\c}} \g_\al$.  We write $\u = \bigoplus_{\al \in \Phi_{\nc}} \g_{\al}$ for the sum of the noncompact root spaces, and so we have $\g = \k \oplus \u$.  Finally, since we will need to distinguish the positive and negative roots, we let $$\Phi^+_{\c} = \Phi^+ \cap \Phi_{\c}, \quad \Phi^+_{\nc} = \Phi^+ \cap \Phi_{\nc}, \quad \Phi^-_{\c} = \Phi^- \cap \Phi_{\c},\quad \Phi^-_{\nc} = \Phi^- \cap \Phi_{\nc}$$ in the obvious way.  Then we have $\u = \u^+ \oplus \u^-$, where
$$
\u^+ = \bigoplus_{\al \in \Phi^+_{\nc}}\g_\al \qquad \text{and} \qquad \u^- = \bigoplus_{\al \in \Phi^-_{\nc}}\g_\al.
$$

Let $W_\g$ and $W_\k$ be the Weyl groups for $\g$ and $\k$ respectively, and let $\ell(w)$ denote the length of a Weyl group element $w \in W_\k$, so that $\ell(w) = |w(\Phi^+_{\c}) \cap \Phi^-_{\c}|$.  Let $\Pi = \{\al_1,\ldots,\al_r\} \subset \Phi^+$ be the set of simple roots, and set $\Pi_{\c} = \Pi \cap \Phi_{\c}$ and $\Pi_{\nc} = \Pi \cap \Phi_{\nc}$.  Then either of the sets $\Pi_{\c}$ or $\Pi_{\nc}$ determines the decomposition $\g=\k \oplus \u$.  The case where $|\Pi_{\nc}|=1$ is especially important in the theory of maximal parabolic subalgebras, and in this case we say that $(G_0,K_0)$ is a \textbf{Hermitian symmetric pair}.  (See \cite{ehp} for a detailed exposition of this theory.)

Let $(-,-)$ denote the Killing form on $\g$, which restricts to a nondegenerate form on $\mathfrak h$ and thus allows us to identify $\mathfrak h$ with $\mathfrak h^*$.  Under this identification, $\al^\vee := \frac{2 \al}{(\al,\al)}$ is identified with the simple coroot in $\mathfrak h$ corresponding to $\al$, for each $\al \in \Pi$.    A weight $\xi \in \mathfrak h^*$ is called an integral weight for $\g$ if $(\xi,\al^\vee) \in \mathbb Z$ for all $\al \in \Pi$; the set of $\g$-integral weights is denoted by $P(\g)$.  The same condition defines the $\k$-integral weights $P(\k)$ if we replace $\Pi$ by $\Pi_{\c}$.  Moreover, $\xi$ is said to be $\g$-dominant if $(\xi,\al) \geq 0$ for all $\al \in \Pi$, and likewise $\k$-dominant if $(\xi,\al) \geq 0$ for all $\al \in \Pi_{\c}$; the sets of dominant integral weights are denoted by $P_+(\g)$ and $P_+(\k)$.  We say that $\xi$ is $\g$-regular if $(\xi,\al) \neq 0$ for all $\al \in \Phi$.
For $\delta \in P_+(\k)$ we let $L_{\k}(\delta)$ denote the finite-dimensional representation of $\k$ with highest weight $\delta$. We distinguish the elements $$\rho_{\c}=\frac{1}{2}\sum_{\al \in \Phi^+_{\c}}\al \qquad \text{and} \qquad \rho_{\nc}=\frac{1}{2}\sum_{\al \in \Phi^+_{\nc}}\al$$ to be half the sum of the positive compact (respectively, noncompact) roots. Finally, define $Q(\xi)$ to be the number of ways of writing $\xi$ as a sum of \textit{non}compact positive roots; in other words, $Q$ is the $\Phi^+_{\nc}$-partition function.

For $\delta,\eta \in P(\k)$, \textbf{Blattner's formula} is
$$
B(\delta,\eta) = \sum_{w \in W_\k} (-1)^{\ell(w)} Q\big(w(\delta + \rho_\c)-\rho_c - \eta\big).
$$
Under certain assumptions on $\delta$ and $\eta$, the output of Blattner's formula gives a $\k$-type multiplicity in a discrete series representation of $G_0$.  In order to index these representations, we appeal to Vogan's theory (see \cite{vogan}) of the \textbf{lowest $\k$-type}, i.e., the unique $\k$-type which appears with multiplicity 1 in the $\k$-decomposition of a discrete series representation.  For $\eta \in P_+(\k)$, we write $\mathcal D^\eta$ for the discrete series representation of $G_0$ whose lowest $\k$-type is $L_\k(\eta)$.  This $\eta$ is sometimes called the ``Blattner parameter'' of the representation $\mathcal D^\eta$.  (See \cite{knapp}, page 310.)  We now state the precise interpretation of Blattner's formula in terms of $\k$-type multiplicities, as proved in \cite{hs}.

\begin{thm}[Hecht and Schmid, 1975]
\label{thm:blattner}
Assume $\delta,\eta \in P_+(\k)$ such that $\eta + \rho_\c - \rho_{\nc}$ is $\g$-dominant regular.  Then $B(\delta,\eta)$ equals the multiplicity of $L_\k(\delta)$ in $\mathcal D^\eta$.
\end{thm}

\subsection*{A formal series and its generating function}  The primary tool used by Willenbring and Zuckerman in \cite{wz} is the formal power series $$
\b(\delta) := \sum_{\eta \in \mathfrak h^*} B(\delta,\eta)e^\eta
$$
which records the values of Blattner's formula when the first argument $\delta \in P_+(\k)$ is held fixed.  Their main result is the generating function
$$
\b(\delta) = \ch L_\k(\delta) \cdot \frac{\prod_{\alpha \in \Phi^+_\c}1-e^{-\alpha}}{\prod_{\alpha \in \Phi^+_{\nc}} 1 - e^{-\alpha}}
$$
with the specialization
\begin{equation}
\label{b0}
    \b(0) = \frac{\prod_{\alpha \in \Phi^+_\c}1-e^{-\alpha}}{\prod_{\alpha \in \Phi^+_{\nc}} 1 - e^{-\alpha}}.
\end{equation}
Our main result will interpret coefficients in $\b(0)$, within the context of Section \ref{s:2ncroots}, as the integers $\varepsilon^\la_{\mu,\nu}$  from the Enright resolutions in Definition \ref{def:epsilon}.  (The reader who is impatient to see this result may proceed to Section \ref{s:2ncroots}, which depends only tangentially on the combinatorics in the next two sections.)

\section{Discrete series and Howe duality: one noncompact simple root} 
\label{s:HermSymm}
We return to the concrete setting in which $G_0$, $K_0$, $\g$, $\k$, etc.\ denote the specific objects in Section \ref{s:classicalinvariant}. Note that $$\text{rank }G_0 = \text{rank }\SU(p,q) = p+q-1 = \text{rank }\S(\U(p) \times \U(q)) = \text{rank }K_0,$$
so that $\SU(p,q)$ satisfies Harish--Chandra's criterion and thus has discrete series representations.  As it turns out, we have already encountered one example of discrete series representations in the Howe decomposition \eqref{howedecomp}.  Recall that inside the stable range $n\geq p+q$, the infinite-dimensional $\widetilde F^\la_{p,q}$ decomposes as a $\k$-module as in \eqref{kdecomp}.  In this case, $\widetilde F^\la_{p,q}$ is a discrete series representation (or limit thereof) of $\SU(p,q)$, and by \eqref{kdecomp} its lowest $\k$-type is $F^{\la^+ }_p \otimes F^{\la^-}_q$. (See \cite{htw}, Theorem 3.2(b).) Therefore, inside the stable range, by regarding $\la$ as the pair $(\la^+,\la^-)$ of highest weights, we can write
$$
\widetilde F^\la_{p,q} \cong \mathcal D^\la
$$
as a discrete series representation of $\SU(p,q)$.
Since in the setting of Section \ref{s:classicalinvariant} the embedding $\k \subset \g$ is block diagonal in nature (see Figure \ref{fig:1nc}), there is exactly one noncompact simple root, namely $\alpha_p$, and so $(G_0,K_0)$ is a Hermitian symmetric pair.  

\begin{figure}[h]
\captionsetup[subfigure]{labelformat=empty}
 
\begin{subfigure}[t]{0.28\textwidth}
\includegraphics[width=0.9\linewidth]{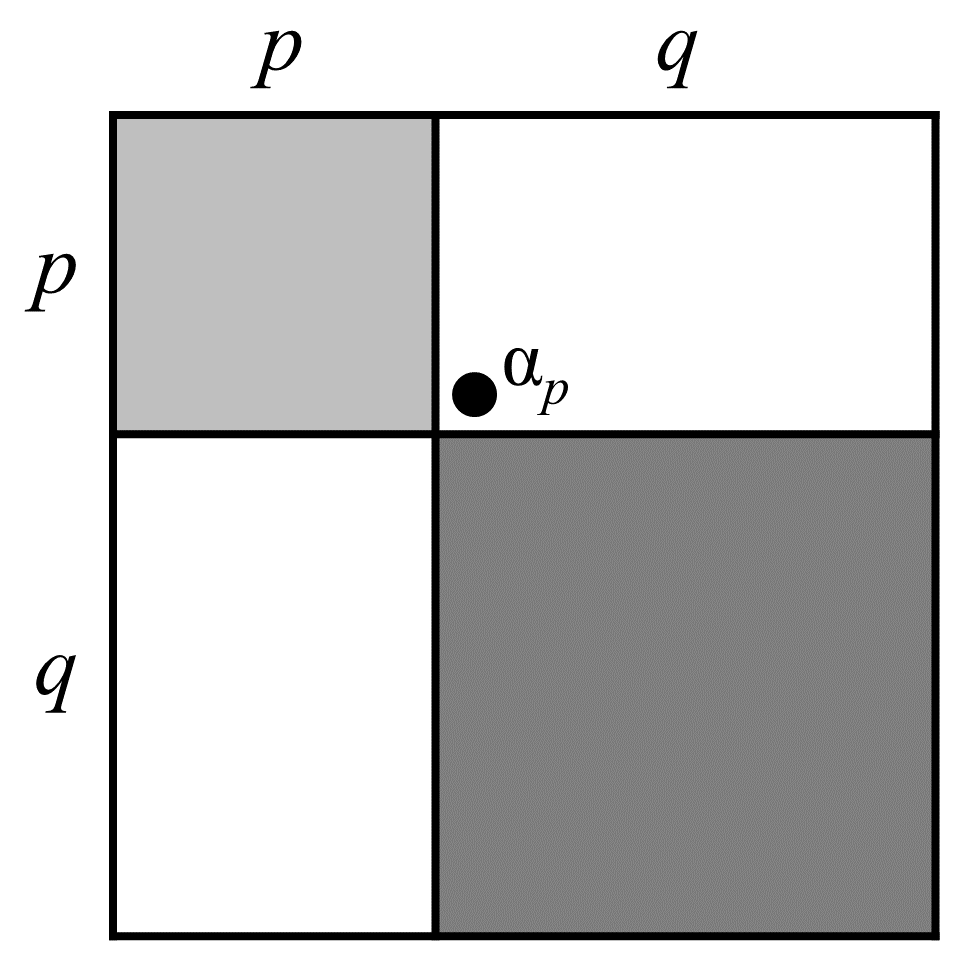}
\caption{$\k=\s(\gl_p \oplus \gl_q)$}
\end{subfigure}
\qquad
\begin{subfigure}[t]{0.28\textwidth}
\includegraphics[width=0.9\linewidth]{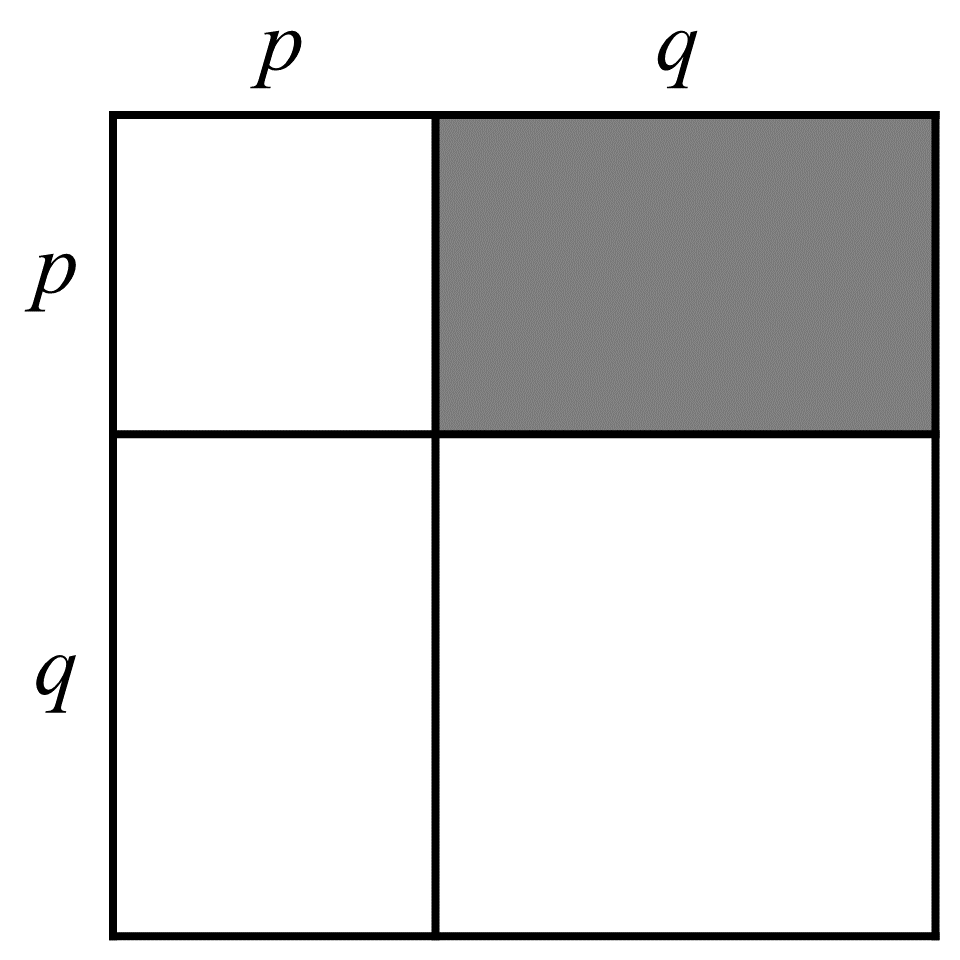}
\caption{$\u^+ = \M_{p,q} $}
\end{subfigure}

\caption{The Hermitian symmetric case (i.e., one noncompact simple root) in Type A, where $\g = \sl_{p+q}$.}
\label{fig:1nc}
\end{figure}

In this Hermitian symmetric case for $G_0 = \SU(p,q)$, we can derive a a combinatorial expression, \textit{without} signs, for Blattner's formula.  Since $\k=\s(\gl_p \oplus \gl_q)$, the weights $\delta$ and $\eta$ in Blattner's formula are actually pairs $\delta=(\delta^p,\delta^q)$ and $\eta=(\eta^p,\eta^q)$.  Here $\delta^p$ and $\eta^p$ are partitions of length at most $p$ (thus highest weights of polynomial representations of $\GL_p$), while $\delta^q$ and $\eta^q$ are partitions of length at most $q$ (and thus highest weights of polynomial representations of $\GL_q$).  Recall that for partitions $\al,\be,\gamma$, with lengths at most $k$, the \textbf{Littlewood--Richardson coefficient} $c^\gamma_{\al,\be}$ is the multiplicity of $F^\gamma_k$ in the tensor product $F^\al_k \otimes F^\be_k$, under the restriction of $\GL_k \times \GL_k$ to its diagonal subgroup $\Delta(\GL_k) \cong \GL_k$. 

\begin{prop}
\label{prop:hermitiancomb}
Let $\g = \sl_{p+q}$, with $\Pi_{\nc} = \{\al_p\}$; hence $\k = \s(\gl_p \oplus \gl_q)$ embedded block diagonally in $\g$.  Let $\delta=(\delta^p,\delta^q)$, $\eta = (\eta^p,\eta^q) \in P_+(\k)$.  Then
\begin{equation}
\label{hermitiancomb}
    B(\delta,\eta) = \sum_\xi c_{\xi, \eta^p}^{\delta^p} c_{\xi, \eta^q}^{\delta^q}
\end{equation}
where the sum is over all partitions $\xi$ such that $\ell(\xi) \leq \min(p,q)$.
\end{prop}

\begin{proof}
By Cauchy's identity, we have the following decomposition as a $K$-module:
\begin{equation}
     \label{cauchy}
     \C[M_{p,q}] \cong \bigoplus_{\xi} F_p^\xi \otimes F_q^\xi,
 \end{equation}
where $\xi$ ranges over all partitions of length at most $\min(p,q)$.  Starting from \eqref{kdecomp}, and putting $\eta^p = \eta^+$ and $\eta^q = \eta^-$, we have
\begin{align*}
     \mathcal D^\eta \cong \widetilde F^\eta_{p,q} & \cong \C[M_{p,q}] \otimes \left(F^{\eta^p}_p \otimes F^{\eta^q}_q \right)\\
    & \cong \left(\bigoplus_{\xi} F_p^\xi \otimes F_q^\xi\right) \otimes \left(F^{\eta^p}_p \otimes F^{\eta^q}_q \right)\\
     & \cong \bigoplus_\xi \left(F_p^\xi \otimes F^{\eta^p}_p \right) \otimes \left( F^\xi_q \otimes F^{\eta^q}_q\right)\\
     & \cong \bigoplus_\xi \left( \bigoplus_{\delta^p} c^{\delta^p}_{\xi, \eta^p} F_p^{\delta^p}\right) \otimes \left( \bigoplus_{\delta^q} c^{\delta^q}_{\xi, \eta^q} F_q^{\delta^q}\right)\\
     &\cong \bigoplus_{\delta^p,\delta^q} \left( \sum_\xi c^{\delta^p}_{\xi, \eta^p} c^{\delta^q}_{\xi, \eta^q}\right) F_p^{\delta^p} \otimes F_q^{\delta^q}\\
     & = \bigoplus_{\delta^p,\delta^q} B(\delta,\eta) \, F_p^{\delta^p} \otimes F^{\delta^q}_{q,}
\end{align*}
where $\delta^p$ and $\delta^q$ range over all partitions of length at most $p$ and $q$ respectively.  Comparing the coefficients in the last two lines, we see that the proposition holds.
\end{proof}

One benefit of Proposition \ref{prop:hermitiancomb} is that there are no negative signs in the sum on the right-hand side of \eqref{hermitiancomb}, whereas Blattner's formula itself is an alternating sum.  One application of our main result will be a similar combinatorial expression, without signs, for discrete series representations \textit{outside} the Hermitian symmetric case (in particular, with \textit{two} noncompact simple roots).  For this, we will need a generalization of the classical Littlewood--Richardson coefficients, which we present in the following section.

\section{Generalized Littlewood--Richardson coefficients}
\label{s:LRC}

In the Howe decomposition \eqref{howedecomp}, the infinite-dimensional modules $\widetilde F^\la_{p,q}$ actually provide combinatorial information about the representation theory of the finite-dimensional modules $F^\la_n$.  Using the method of seesaw pairs, the authors of \cite{cew} describe a generalization of the Littlewood--Richardson coefficients $c^\gamma_{\al, \be}$ which applies to \textit{multi}-tensor product multiplicities for \textit{rational} (not only polynomial) representations of $\GL_n$.  Below we outline the generalization in the case of only two tensor factors, since this will suffice for the present paper.  

Let $\al, \be, \gamma$ be $n$-tuples of weakly decreasing integers (i.e., highest weights indexing rational $\GL_n$-representations).  We wish to obtain the multiplicity of $F^\gamma_n$ inside $F^\al_n \otimes F^\be_n$, which we will denote by $\LRC{\gamma}{\al,\be}$.  To this end, we consider all possible ``hollow contingency tables'' of the form
\begin{equation*}
\begin{blockarray}{c c c c}
& \al^- & \be^- & \gamma^+\\
\begin{block}{c [c c c]}
\al^+ & 0 & * & *\\
\be^+ & * & 0 & *\\
\gamma^- & * & * & 0\\ \end{block}
\end{blockarray}
\end{equation*}
where each star denotes an arbitrary partition (0 being the empty partition, corresponding to the trivial $\GL_n$-representation).  Note that each row and column naturally corresponds to a (classical) Littlewood--Richardson coefficient; for example, the first row corresponds to $c^{\al^+}_{*,*}$.  To obtain the multiplicity $\LRC{\gamma}{\al,\be}$, multiply these six Littlewood--Richardson coefficients together, and then sum over all such tables.

\begin{exam}
Let $\al=(1,0,\ldots,0)$ and $\be = (0,\ldots,0,-1)$.  Hence $F^\al_n \cong \C^n$ is the defining representation of $\GL_n$, while $F^\be_n$ is its dual $(\C^n)^*$.  Then $F^\al_n \otimes F^\be_n \cong \M_n(\C) \cong \gl_n$, the adjoint representation of $\GL_n$.  We have $\al^+ = (1)$, $\al^- = 0$, $\be^+ = 0$, and $\be^- = (1)$. Treating $\gamma$ as the unknown, there are only two ways of filling the contingency table so that the product of Littlewood--Richardson coefficients is nonzero, namely
$$
\begin{blockarray}{c c c c}
& 0 & (1) & \gamma^+\\
\begin{block}{r [c c c]}
(1) & 0 & 0 & (1)\\
 0 & 0 & 0 & 0\\
 \gamma^- & 0 & (1) & 0\\ \end{block}
 \end{blockarray}
 \qquad \text{and} \qquad
 \begin{blockarray}{c c c c}
 & 0 & (1) & \gamma^+\\
 \begin{block}{r [c c c]}
 (1) & 0 & (1) & 0\\
 0 & 0 & 0 & 0\\
 \gamma^- & 0 & 0 & 0\\ \end{block}
 \end{blockarray}.
 $$
 The first table forces $\gamma=[(1),(1)]=(1,0,\ldots,0,-1)$, and the second table forces $\gamma = 0$; in each case, by a simple application of the Pieri rule, all six Littlewood--Richardson coefficients are 1.  Hence $F^\al_n \otimes F^\be_n = F^{(1,0,\ldots,0,-1)}_n \oplus F^0_n$, or in other words, the adjoint representation $\gl_n$ is the direct sum $\sl_n \oplus \C I$.
\end{exam}

As an immediate advantage, which the reader can check, we can rewrite \eqref{hermitiancomb} as a single generalized Littlewood--Richardson coefficient:
$$
B(\delta,\eta) = \LRC{[\eta^p,\eta^q]}{\delta^p,\:[0,\delta^q]}.
$$
We will use these generalized coefficients again in Theorem \ref{thm:LRC2nc} to write down a combinatorial expression for Blattner's formula with two noncompact simple roots.

\section{Discrete series and Howe duality: two noncompact simple roots}
\label{s:2ncroots}

It turns out that the classical decomposition in \eqref{howedecomp} is related, in a very different way, to certain discrete series representations \textit{outside} the Hermitian symmetric case outlined in Section \ref{s:HermSymm}.  In fact, the thrust of our main result is that we can read off the character theory of $\C[V]$ --- in particular the integers $\varepsilon^\la_{\mu,\nu}$ and hence the Enright resolutions of the modules $\widetilde F^\la_{p,q}$ --- from the formal series $\b(0)$ for the group $\SU(n,p+q)$ with \textit{two} noncompact simple roots.  To distinguish this new setting from the Hermitian symmetric setting of Section \ref{s:HermSymm}, we will now decorate all of our notation with prime symbols; a complete summary is found in Table \ref{table:notation}.  See also Figure \ref{fig:2nc}.  
\begin{table}[ht]
\caption{}
\centering
\begin{tabular}{ |c||c| c||c |c|}
 \hline
 & \multicolumn{2}{c ||}{Sections \ref{s:classicalinvariant} and \ref{s:HermSymm}} & \multicolumn{2}{c|}{Section \ref{s:2ncroots}} \\ \hline \hline
 Classical representation & $V$ & $\M_{n,p} \oplus \M_{n,q}$ & -- & -- \\ \hline
 
 Group & $G_0$ & $\SU(p,q)$ & $G'_0$ & $\SU(n,p+q)$ \\ \hline
 Compl.\ Lie algebra & $\g$ & $\sl_{p+q}$ & $\g'$ & $\sl_{n+p+q}$\\ \hline
 
 Nc simple roots & $\Pi_{\nc}$ & $\{\al_p\}$ & $\Pi'_{\nc}$ & $\{\al_p,\al_{p+n}\}$ \\ \hline
 
 Compl.\ max. compact & $\k$ & $\s(\gl_p \oplus \gl_q)$ & $\k'$ & $\s(\gl_n \oplus \gl_{p+q})$ \\ \hline
 
 Levi subalgebra & -- & -- & $\m'$ & $\s(\gl_n \oplus \gl_p \oplus \gl_q)$ \\ \hline
 
  $\bigoplus$ pos nc root spaces & $\u^+$ & $\M_{p,q}$ & $\u'^+$ & $\M_{p,n} \oplus \M_{n,q}$\\ \hline
  
  $\bigoplus$ neg nc root spaces & $\u^-$ & $\M_{q,p}$ & $\u'^-$ & $\M_{n,p} \oplus \M_{q,n}$\\ \hline
 
\end{tabular}
\label{table:notation}
\end{table}

\begin{figure}[h]
\captionsetup[subfigure]{labelformat=empty}
 
\begin{subfigure}[t]{0.28\textwidth}
\includegraphics[width=0.9\linewidth]{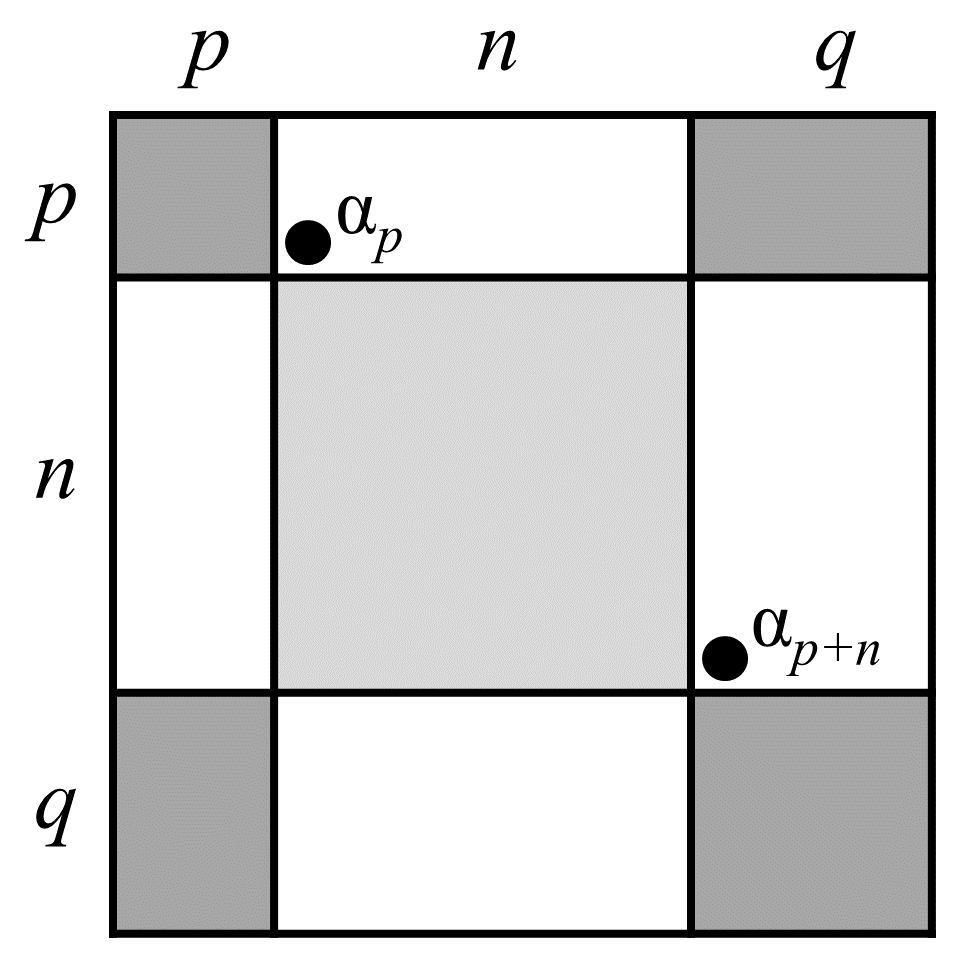} 
\caption{$\k' = \s(\gl_n \oplus \gl_{p+q})$}
\end{subfigure}
\hfill
\begin{subfigure}[t]{0.28\textwidth}
\includegraphics[width=0.9\linewidth]{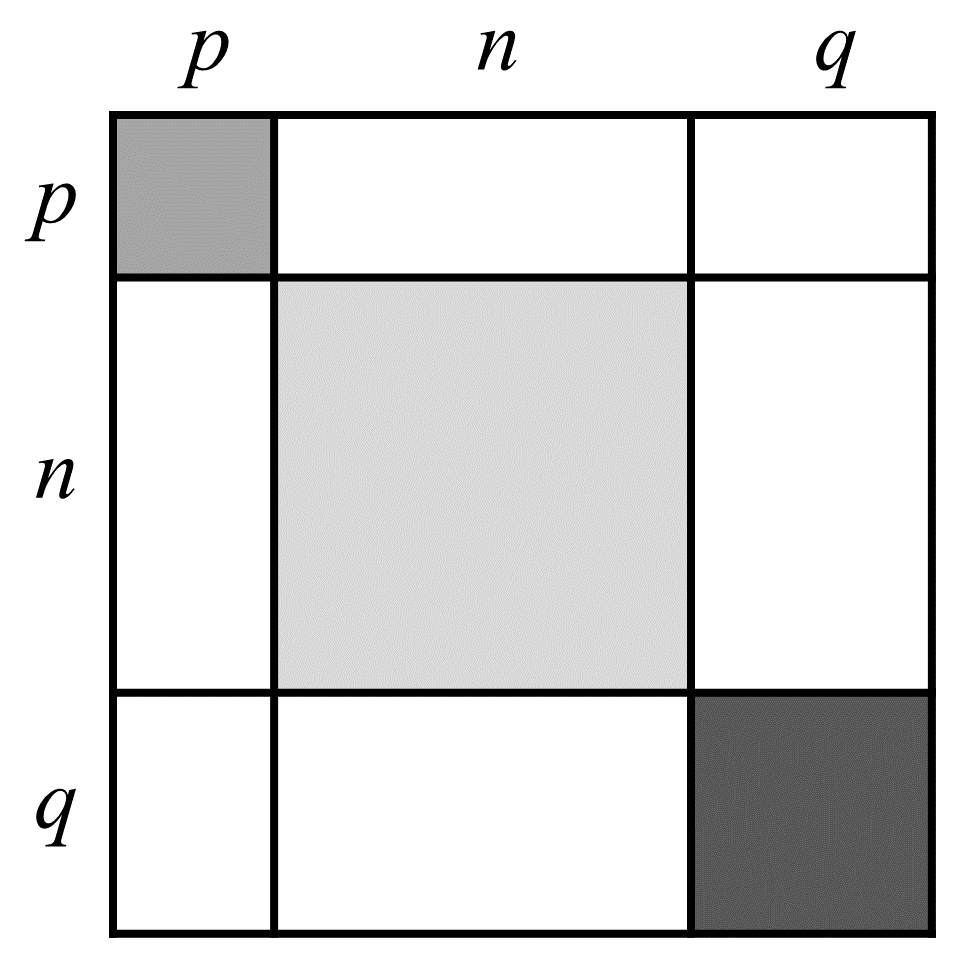}
\caption{$\m'=\s(\gl_n \oplus \gl_p\oplus \gl_q)$}
\end{subfigure}
\hfill
\begin{subfigure}[t]{0.28\textwidth}
\includegraphics[width=0.9\linewidth]{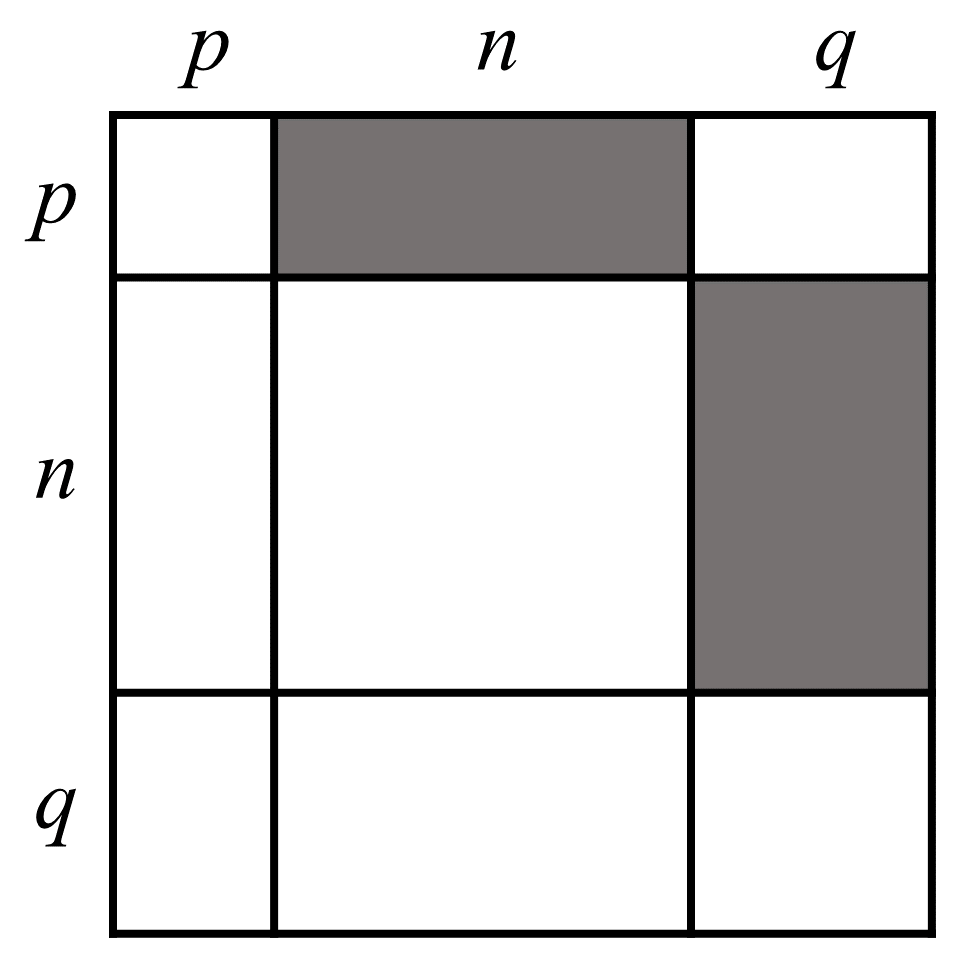}
\caption{$\u'^+ = \M_{p,n} \oplus \M_{n,q}$}
\end{subfigure}

\caption{A visual companion to the notation in Table \ref{table:notation} for Section \ref{s:2ncroots}; regard as subspaces of $\g' = \sl_{n+p+q}$.}
\label{fig:2nc}
\end{figure}

We re-approach Blattner's formula in a seemingly unrelated context.  Let $G'_0 = \SU(n,p+q)$, with the complexified Lie algebra $\g' = \sl_{n+p+q}$.  Suppose that $\g'$ has two noncompact simple roots, namely $\Pi'_{\nc} = \{ \al_p, \al_{p+n} \}$.  This implies that $\g'$ has maximal compact subalgebra $\k' = \s(\gl_n \oplus \gl_{p+q})$, where the $\gl_n$ is embedded in the ``middle'' and $\gl_{p+q}$ is embedded in the ``four corners.''  Hence the direct sum of the positive noncompact root spaces $\u'^+ = \bigoplus_{\al \in \Phi '^+_{\nc}} \g'_\al$ is embedded in $\g'$ as the two blocks in the upper-right, so that $\u'^+ \cong \M_{p,n} \oplus \M_{n,q}$ as a vector space.  Likewise, the sum $\u'^-$ of the negative noncompact root spaces is embedded in the two corresponding blocks in the lower-left, so that $\u' \cong \M_{n,p} \oplus \M_{q,n}$ as a vector space.  Special importance will be played by the Levi subalgebra $\m' = \s(\gl_n \oplus \gl_p \oplus \gl_q) \subset \k'$, which is the Lie algebra of $M' = \S(\GL_n \times \GL_p \times \GL_q)$.

We will write a weight of $\g'$ as a triple 
$$\llbracket\la,\mu,\nu\rrbracket := (\underbrace{-\mu_p,\ldots,-\mu_1}_{\mu^*},\underbrace{\la_1,\ldots,\la_n}_\la,\underbrace{\nu_1,\ldots,\nu_q}_\nu)
$$
where $\la \in P(\gl_n)$, $\mu \in P(\gl_p)$, and $\nu \in P(\gl_q)$.  The resulting $(n+p+q)$-tuple on the right-hand side is written in terms of the standard coordinates $\varepsilon_i:{\rm diag}[h_1,\ldots,h_{n+p+q}] \longmapsto h_i$ on $\g'$.  Notice that the order of the three weights in $\llbracket\la,\mu,\nu\rrbracket$ follows our usual alphabetical order $n,p,q$, whereas the explicit tuple transposes $\la$ and $\mu$ in order to respect the embedding of $\m'$ in $\g'$.  The reason for the dual on $\mu$ will soon become apparent, in the proof of Lemma \ref{lemma:deltabarcancel}: in order to line up the actions of $M'$ in the Howe duality setting and in the Blattner setting, we will need to regard the summand $\gl_p$ as being embedded in $\g'$ with a twist (i.e., negative transpose). 

A weight $\llbracket\la,\mu,\nu\rrbracket \in P(\g')$ is $\k'$-dominant if and only if \begin{equation*}
\la_1 \geq \cdots \geq \la_n \qquad \text{and} \qquad -\mu_p \geq \cdots \geq -\mu_1 \geq \nu_1 \geq \cdots \geq \nu_q.
\end{equation*}  We will need this fact only in the proof of Theorem \ref{thm:LRC2nc}. For most of this paper, we will actually be concerned with $\m'$-dominant weights, in which we drop the condition $-\mu_1 \geq \nu_1$; in other words, $\la \in P_+(\gl_n)$ and $\mu \in P_+(\gl_p)$ and $\nu \in P_+(\gl_q)$.  Note that for any $\la,\mu,\nu$ satisfying the conditions \eqref{howeconditions}, the weight $\llbracket\la,\mu,\nu\rrbracket$ is automatically $\m'$-dominant.

\begin{figure}[h]
\begin{subfigure}[t]{0.28\textwidth}
\includegraphics[width=0.9\linewidth]{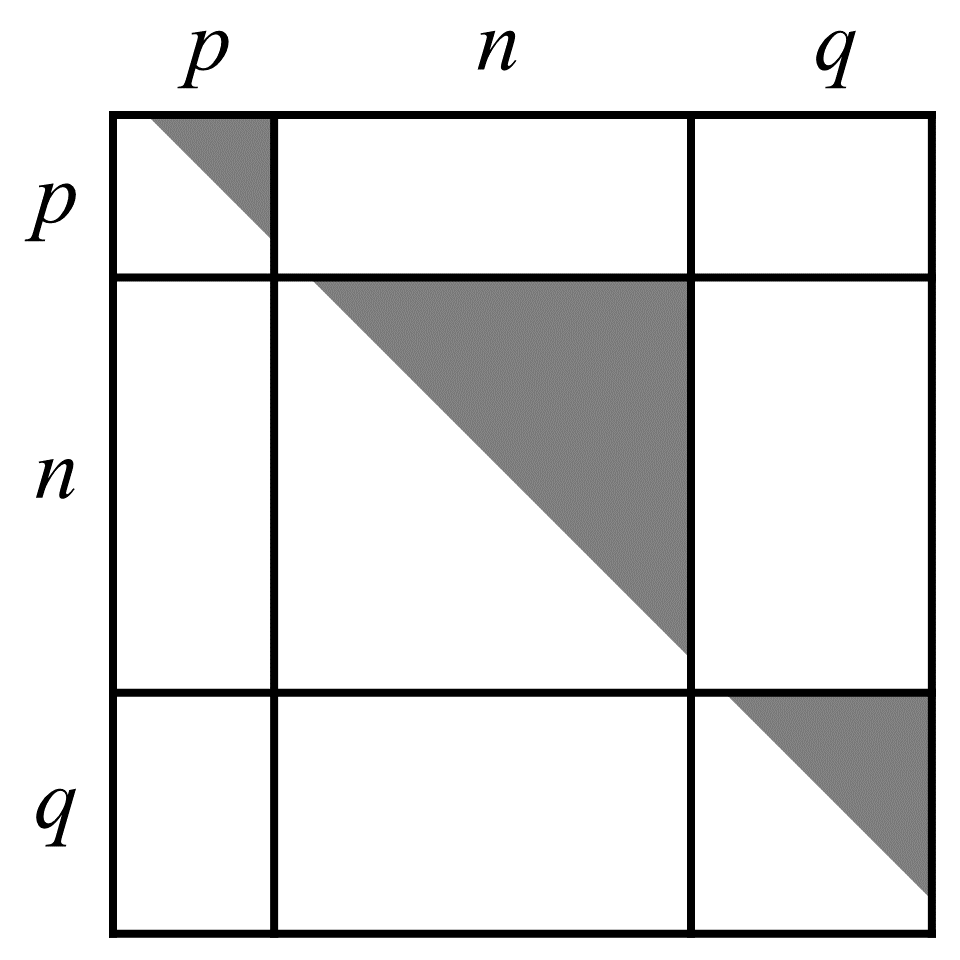} 
\caption{$\Phi_{\m'}$}
\label{subfig:phim}
\end{subfigure}
\hfill
\begin{subfigure}[t]{0.28\textwidth}
\includegraphics[width=0.9\linewidth]{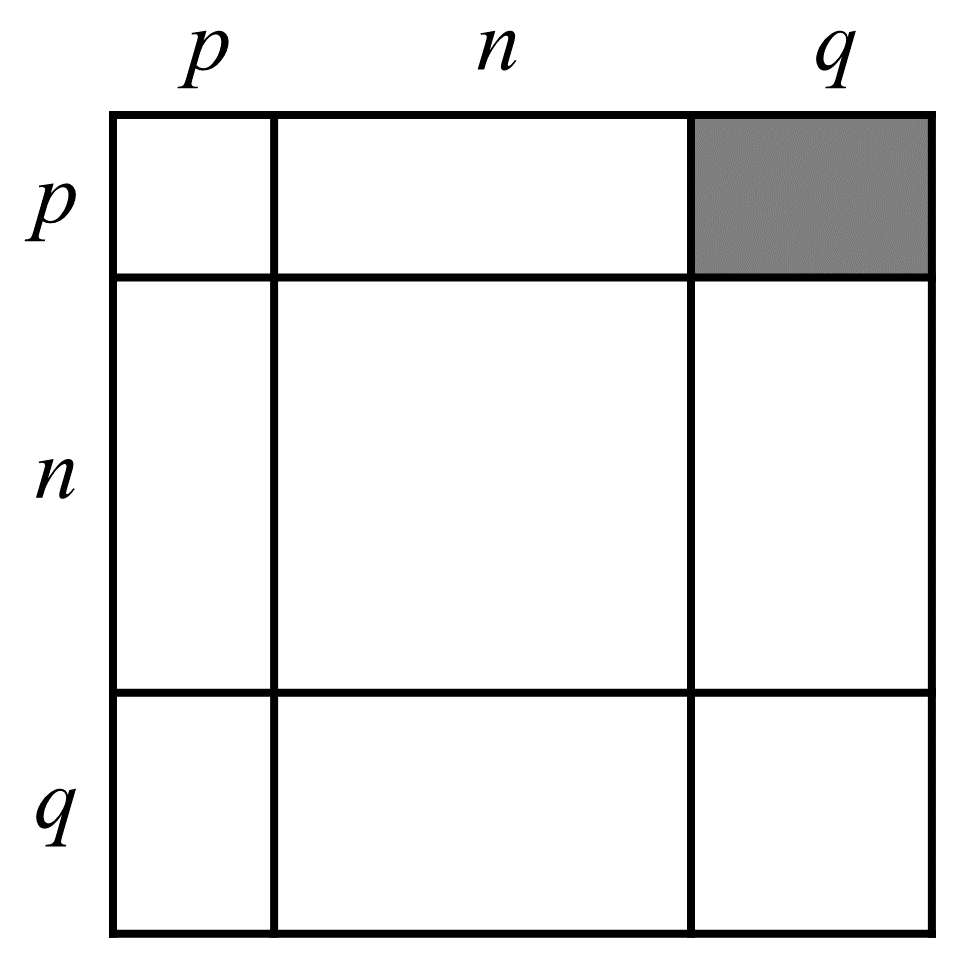}
\caption{$\overline{\Phi_{\m'}}$}
\label{subfig:phimbar}
\end{subfigure}
\hfill
\begin{subfigure}[t]{0.28\textwidth}
\includegraphics[width=0.9\linewidth]{u.PNG}
\caption{$\Phi'^+_{\nc}$}
\label{subfig:phinc}
\end{subfigure}

\caption{Root spaces corresponding to three subsets of $\Phi'^+$.}
\label{fig:b0numden}
\end{figure}

With the aim of rewriting $\b(0)$ in a helpful way, we now partition the positive compact roots $\Phi_{\c}'^+$ of $\g'$ into two subsets, as in Figures \ref{subfig:phim} and \ref{subfig:phimbar}. Let $\Phi_{\m'}=\{\al \in \Phi'^+_{\c} \mid \g'_\al \subset \m'\}$, which contains those positive compact roots whose root spaces span the three triangular regions in Figure \ref{subfig:phim}.  Let $\overline{\Phi_{\m'}}$ denote the complement of $\Phi_{\m'}$ in $\Phi '^+_{\c}$, which contains those positive compact roots whose root spaces span the upper-right $p\times q$ block in Figure \ref{subfig:phimbar}.  Now define the products
$$\displaystyle \Delta_{\m'} =  \prod_{\al \in \Phi_{\m'}} 1-e^{-\al}, \qquad \displaystyle \overline{\Delta_{\m'}} = \prod_{\al \in \overline{\Phi_{\m'}}} 1-e^{-\al}, \qquad \displaystyle \Delta_{\u'^+}= \prod_{\al \in \Phi'^+_{\nc}} 1-e^{-\al},$$
where the ``$\Delta$'' notation is meant to evoke the Weyl denominator from the Weyl character formula.  Now we can rewrite \eqref{b0} as
\begin{equation}
\label{b0new}
    \b(0) = \frac{\prod_{\alpha \in \Phi^+_\c}1-e^{-\alpha}}{\prod_{\alpha \in \Phi^+_{\nc}} 1 - e^{-\alpha}} = \frac{\Delta_{\m'} \cdot \overline{\Delta_{\m'}}}{\Delta_{\u'^+}}.
\end{equation}

The following two lemmas capture the connection between the expression for $\b(0)$ in \eqref{b0new} and the classical invariant setting from Section \ref{s:classicalinvariant}.

\begin{lemma}
\label{lemma:deltabarcancel}
Consider $\C[\M_{p,q}]$ from the classical invariant setting in Section \ref{s:classicalinvariant}.  Then  $\overline{\Delta_{\m'}} = \big(\ch \C[\M_{p,q}]\big)^{-1}$ as a character of $\m'$.
\end{lemma}

\begin{proof}
Let $\psi:\M_{p,q} \longrightarrow \displaystyle\bigoplus_{\al \in \overline{\Phi_{\m'}}} \g'_\al$ be the block ``upper-right'' embedding given by $\psi(E_{i,j}) = E_{i,p+n+j}$, for $i=1,\ldots,p$ and $j = 1, \ldots, q$.  (Here $E_{i,j}$ denotes the $p \times q$ matrix with a 1 as the $(i,j)$ entry and 0's elsewhere.)  Clearly $\psi$ is a vector space isomorphism; we claim that $\psi$ is in fact an isomorphism of $\m'$-modules.

To see this, recall the action of $K=\S(\GL_p \times \GL_q)$ on $\M_{p,q}$ given in \eqref{kactiononMpq} in the Howe duality setting, where $(g,h)\cdot X = g^{-T}Xh^{-1}$.  Here in the Blattner setting, we also have $K \subset M' \subset G'=\SL_{n+p+q}$ embedded block-diagonally with $\GL_p$ in the upper-left and $\GL_q$ in the lower-right.  In this embedding, ${\rm im}(\psi) \subset \g'$ is a $K$-module via the adjoint action, and the explicit action is $(g,h)\cdot X = gXh^{-1}$. Therefore, the $K$-actions in the Howe duality setting and in the Blattner setting are the same, up to a twist in the $\GL_p$-action. (This is remedied by embedding $\GL_p$ in the upper-left via its inverse transpose, and at the Lie algebra level, by embedding $\gl_p$ via its negative transpose.) Extending this $K$-action to $M'$ by letting the factor $\GL_n$ act trivially, we conclude that $\M_{p,q} \cong {\rm im}(\psi)$ as modules for $M'$, and thus for $\m'$, which proves the claim.

Now, ${\rm im}(\psi)^T:= \{X^T \mid X \in {\rm im}(\psi)\}$ is the embedding of $\M_{p,q}$ via the transpose into the lower-left block of $\g'$.  Furthermore, ${\rm im}(\psi)^T \cong {\rm im}(\psi)^*$ as an $M'$-module, which is clear from the adjoint action of $M' \subset G'$ on $\g'$.  Therefore
\begin{equation*}
    \Sym\big({\rm im}(\psi)^T\big) \cong \C[{\rm im}(\psi)] \cong \C[\M_{p,q}]
\end{equation*}
as $\m'$-modules. Now we can conclude that
\begin{align*}
    \overline{\Delta_{\m'}} &= \prod_{\al \in \overline{\Phi_{\m'}}} (1 - e^{-\al})\\
    &= \Big(\ch \Sym\big({\rm im}(\psi)^T\big)\Big)^{-1} \\
    &= \big(\ch \C[\M_{p,q}]\big)^{-1}
\end{align*}
as a character of $\m'$.
\end{proof}

\begin{lemma}
\label{lemma:V_iso_u'+}
We have $(\Delta_{\u'^+})^{-1} = \ch \C[V]$ as a character of $\m'$.
\end{lemma}

\begin{proof}
Let $\varphi: V \longrightarrow \u'^+$ be given by $\varphi(X,Y) = (X^T,Y)$.  We claim that $\varphi$ is an isomorphism of $\m'$-modules.

To see this, note that an element of $V$ is of the form $(X,Y)$, while an element of $\u'^+$ is of the form $(X^T,Y)$, where $X \in \M_{n,p}$ and $Y \in \M_{n,q}$.  Recall the action of $M'$ on $V$ given in \eqref{Maction}, from the Howe duality setting.  Here in the Blattner setting (as explained in the proof of Lemma \ref{lemma:deltabarcancel}), regard $\GL_p$ as being embedded in the upper-left block of $G'$ via inverse transpose; then the adjoint action of $M' \subset G'$ on $\u'^+ \subset \g'$ is given by
$$
(g_n,g_p,g_q) \cdot (X^T,Y) = (g_p^{-T} X^T g_n^{-1},\: g_n Y g_q^{-1}).
$$ 
For $g = (g_n,g_p,g_q) \in M'$, we must show that $g \circ \varphi = \varphi \circ g$.  But
\begin{align*}
    g \circ \varphi (X,Y) &= g(X^T,Y) \\
    &=(g_p^{-T} X^T g_n^{-1},\: g_n Y g_q^{-1})\\
    &= \varphi (g_n^{-T} X g_p^{-1},\: g_n Y g_q^{-1})\\
    &= \varphi \circ g(X,Y),
\end{align*}
by \eqref{Maction}, which proves the claim.

Now, observing from the adjoint $M'$-action that $\u'^+ \cong (\u'^-)^*$ as $M'$-modules, we have
\begin{equation*}
    \Sym(\u'^-) \cong \C[V]
\end{equation*}
as modules for $M'$, and therefore for $\m'$.  Therefore we have 
\begin{align*}
    \left(\Delta_{\u'^+}\right)^{-1} &=\prod_{\al \in \Phi^+_{\nc}}\frac{1}{1-e^{-\al}}\\
    &=\ch \Sym(\u'^-)\\
    &= \ch \C[V]
\end{align*}
as a character of $\m'$.
\end{proof}

At this point, we should observe that for $\xi \in P_+(\m')$, the Weyl character formula can be written as 
$$
\ch L_{\m'}(\xi) = \frac{\sum_{w \in W_{\m'}}(-1)^{\ell} e^{w(\xi + \rho)}}{e^\rho \prod_{\al \in \Phi_{\m'}} 1-e^{-\al}}  = \frac{\sum_{w \in W_{\m'}}(-1)^{\ell} e^{w(\xi + \rho)-\rho}}{\Delta_{\m'}},
$$
where $W_{\m'}$ is the Weyl group for $\m'$ and $\rho = \frac{1}{2} \sum_{\al \in \Phi_{\m'}} \al$.  Upon rearranging, this says that the product $\Delta_{\m'} \cdot \ch L_{\m'}(\xi)$ is the alternating sum of terms of the form $e^{w(\xi + \rho)-\rho}$.  But $e^{w(\xi + \rho)-\rho} \in P_+(\m')$ if and only if $w = 1$, which means that
\begin{equation*}
    \Delta_{\m'} \cdot \ch L_{\m'}(\xi) = e^\xi + \text{alternating sum of }e^{\text{nondominant $\m'$-weight}}\text{'s.}
\end{equation*}
More generally, consider an arbitrary $\m'$-module $L=\bigoplus_\xi m_\xi L_{\m'}(\xi)$, ranging over $\xi \in P_+(\m')$, with multiplicities $m_\xi \in \mathbb N$.  Then we have
\begin{equation}
\label{weyldenomeffect}
    \Delta_{\m'} \cdot \ch L = \sum_\xi m_\xi e^\xi + \text{alternating sum of }e^{\text{nondominant $\m'$-weight}}\text{'s.}
\end{equation}
The upshot is that multiplying the character of an $\m'$-module $L$ by $\Delta_{\m'}$ produces a sum of formal weights, in which the $\m'$-dominant weights are precisely the highest weights of the irreducible $\m'$-modules in the decomposition of $L$, the coefficients of which are their multiplicities in $L$.  

We have now arrived at our main result, uniting the Enright resolutions of Section \ref{s:classicalinvariant} with the discrete series representations in the present section.
\begin{thm}
\label{thm:mainresult}
 Let $\b(0)$ be the formal series in \eqref{b0}, in the context where the Lie algebra $\g' = \sl_{n+p+q}$ has the two noncompact simple roots $\al_p$ and $\al_{p+n}$.  Let $\la,\mu,\nu$ satisfy the conditions \eqref{howeconditions}, and let $\varepsilon^\la_{\mu,\nu}$ be as in Definition \ref{def:epsilon}.  Then
 $$
 \varepsilon^\la_{\mu,\nu} = \text{the coefficient of }e^{\llbracket\la,\mu,\nu\rrbracket} \text{ in }\b(0).
 $$
\end{thm}

\begin{proof}
From \eqref{b0new}, we have
$$
\b(0) = \frac{\Delta_{\m'} \cdot \overline{\Delta_{\m'}}}{\Delta_{\u'^+}}
$$
which, by Lemmas \ref{lemma:deltabarcancel} and \ref{lemma:V_iso_u'+}, becomes
$$
\b(0)=\Delta_{\m'} \cdot (\ch \C[\M_{p,q}])^{-1} \cdot \ch \C[V].
$$
Substituting for $\ch \C[V]$ from \eqref{bigcharacter}, we find that the two instances of $\ch \C[\M_{p,q}]$ cancel each other out:
\begin{align}
    \b(0) &= \Delta_{\m'} \cdot (\ch \C[\M_{p,q}])^{-1} \cdot \ch \C[\M_{p,q}] \cdot \sum_{\la,\mu,\nu} \varepsilon^\la_{\mu,\nu} \cdot \ch (F_n^\la \otimes F^{\mu}_p \otimes F^\nu_q) \nonumber \\
    &= \Delta_{\m'} \cdot \sum_{\la,\mu,\nu} \varepsilon^\la_{\mu,\nu} \cdot \ch (F_n^\la \otimes F^{\mu}_p \otimes F^\nu_q). \label{b0final}
\end{align}
The sum in the last line is a virtual $\m'$-character, and so by \eqref{weyldenomeffect}, we have
\begin{equation}
\label{b0inproof}
\b(0) = \sum_{\la,\mu,\nu} \varepsilon^\la_{\mu,\nu} \cdot e^{\llbracket\la,\mu,\nu\rrbracket} + \sum_{\xi \not \in P_+(\m')} c_\xi e^\xi
\end{equation}
with $c_\xi \in \mathbb Z$.  This completes the proof.
\end{proof}

\begin{exam}
\label{ex:b0}
We fully work out the case when $n=p=q=1$. Note that this is outside the stable range.  In this case, $\C[V] = \C[x,y]$, and the action of $M'=\S(\GL_1 \times \GL_1 \times \GL_1)$ upon a typical monomial follows from \eqref{Maction}:
$$
(g_n,g_p,g_q) \cdot x^ay^b = \left(g_n g_p x\right)^a \left(\frac{g_q}{g_n} y\right)^b = \left(g_n^{a-b} g_p^a g_q^b\right) x^a y^b,
$$
where $g_n,g_p,g_q \in \C^\times$.
Hence each monomial spans a one-dimensional subrepresentation of $M'$.  Note that by setting $\la = a-b$ and $c = \min(a,b)$, we can rewrite the formal sum of these monomials as
\begin{align*}
\sum_{a,b \in \mathbb N} x^a y^b &= \sum_{c=0}^\infty (xy)^c \left(\sum_{\la \geq 0} x^\la + \sum_{\la < 0} y^{-\la} \right) \\
&= \sum_{\la \geq 0} x^\la \sum_{c=0}^\infty (xy)^c + \sum_{\la < 0} y^{-\la} \sum_{c=0}^\infty (xy)^c.
\end{align*}

When $n=p=q=1$, each rational representation of $\GL_n$ and $\GL_p$ and $\GL_q$ is one-dimensional, indexed by a single integer $\la$ (nonnegative if the representation is polynomial), where the group action is $g \cdot z = g^\la z$.  Hence we will write $\C_n^\la$ for the representation $F_1^\la$ of $\GL_n=\GL_1$, and likewise for $\C_p^\la$ and $\C_q^\la$.  Comparing the two calculations above, we see that the Howe decomposition \eqref{howedecomp} in this case is
$$
\C[x,y] = \bigoplus_{\la \in \mathbb Z} \C_n^\la \otimes \widetilde F^\la_{1,1}
$$
where
$$
\widetilde F^\la_{1,1} = 
\begin{cases}
\C[xy] \otimes \C_p^{\la} \otimes \C_q^0, & \la \geq 0,\\
\C[xy] \otimes \C_p^0 \otimes \C^{-\la}_q,& \la<0.
\end{cases}
$$
Therefore the Enright resolution of $\widetilde F^\la_{1,1}$ contains only one term, and so for $\la,\mu,\nu \in \mathbb Z$ we have
\begin{equation*}
    \varepsilon^\la_{\mu,\nu}=
    \begin{cases}
    1,&\la=\mu\geq 0 \text{ and }\nu=0,\\
    1,&\la = -\nu < 0 \text{ and }\mu = 0,\\
    0 & \text{otherwise},
    \end{cases}
\end{equation*}
so that the triples $(\la,\mu,\nu)$ satisfying the first two cases are of the form 
\begin{equation}
\label{supportinexample}
    \mathbb N(1,1,0) \text{ and } \mathbb N(-1,0,1).
\end{equation}

Next we compute $\b(0)$ in order to check it against \eqref{supportinexample}.  We have $\g' = \sl_3$, with $\Pi'_{\nc} = \{\al_1,\al_2\} = \Pi'$.  Then $\Phi'^+_{\c}=\{\varepsilon_1 - \varepsilon_3\}$ and $\Phi'^+_{\nc} = \{\varepsilon_1 - \varepsilon_2, \varepsilon_2 - \varepsilon_3\}$. Recall that the triple $\llbracket \ell, m, n \rrbracket$ corresponds to the $\g'$-weight $(-m,\ell,n)$ in standard $\varepsilon$-coordinates.  Setting $t_i = e^{\varepsilon_i}$ for $i=1,2,3$, we have
\begin{align*}
\b(0) = \frac{1-\frac{t_3}{t_1}}{\left(1-\frac{t_2}{t_1}\right)\left(1-\frac{t_3}{t_2}\right)} &= \frac{1}{1-\frac{t_2}{t_1}} + \frac{\frac{t_3}{t_2}}{1-\frac{t_3}{t_2}}\\
&= \sum_{k=0}^\infty e^{k(\varepsilon_2 - \varepsilon_1)} + \sum_{k=1}^\infty e^{k(\varepsilon_3 - \varepsilon_2)}\\
&=\sum_{k=0}^\infty e^{k(-1,1,0)} + \sum_{k=1}^\infty e^{k(0,-1,1)}\\
&=\sum_{k=0}^\infty e^{k\llbracket 1,1,0\rrbracket} + \sum_{k=1}^\infty e^{k\llbracket -1,0,1\rrbracket},
\end{align*}
coinciding exactly with \eqref{supportinexample} and Theorem \ref{thm:mainresult}.  (See Figure \ref{fig:b0example} for a visualization of $\b(0)$, in which we plot the support of $B(0,-)$ on the weight lattice of $\sl_3$.)

\begin{figure}[ht]
    \centering
    \includegraphics[width=0.4\linewidth, frame]{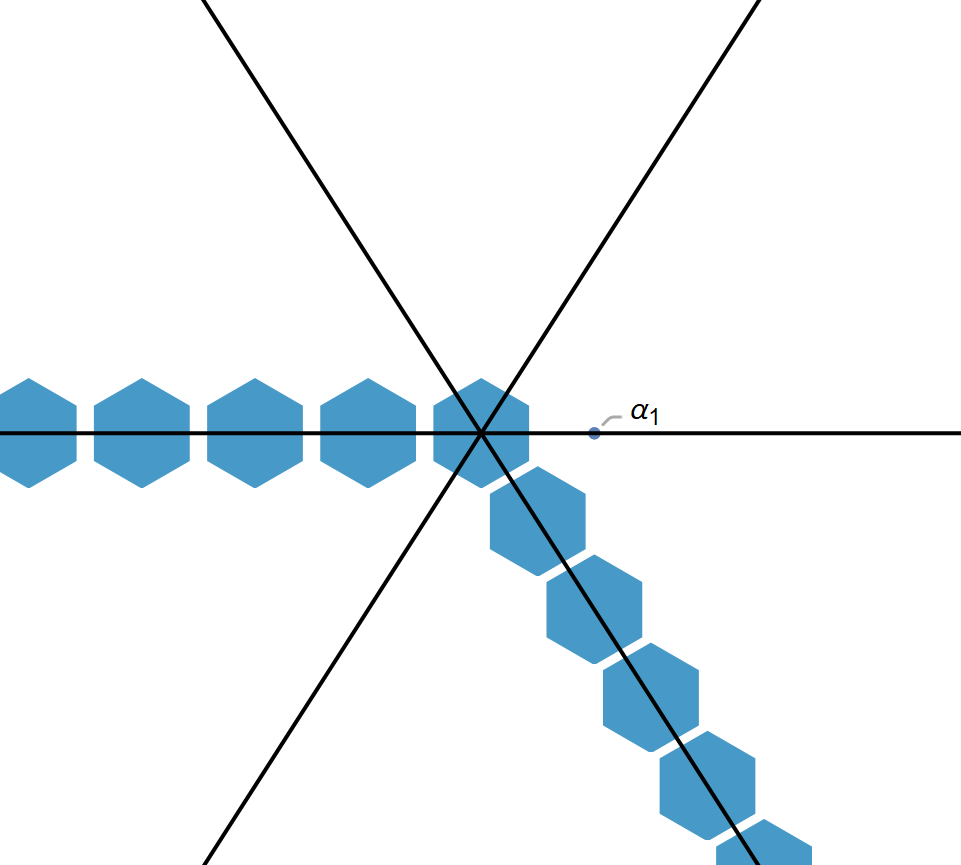}
    \caption{A visualization of $\b(0)$ from Example \ref{ex:b0}.  By programming Blattner's formula directly in Mathematica, we plot a hexagon at each weight $\eta \in P(\sl_3)$ such that $B(0,\eta)=1$.  Note that these weights are the nonnegative multiples of $-\al_1=(-1,1,0)$ and of $-\al_2=(0,-1,1)$.  (All other weights return 0.)}
    \label{fig:b0example}
\end{figure}
\end{exam}

Before presenting a more interesting example, we outline how our results will allow us to write down explicitly the Enright resolution for $\widetilde F^\la_{p,q}$. Separating the first sum in \eqref{b0inproof}, we can write
\begin{equation}
    \label{b0separate}
\b(0)=\sum_\la e^{\llbracket \la,0,0\rrbracket} \underbrace{\left(\sum_{\mu,\nu} \varepsilon^\la_{\mu,\nu} \cdot e^{\llbracket 0,\mu,\nu\rrbracket} \right)}_{:=\text{ ``coefficient'' of $e^\la$}} + \sum_{\xi \not \in P_+(\m')} c_\xi e^\xi.
\end{equation}
Therefore we begin by computing $\b(0)$, expanding to a sufficiently high order, and then ignoring all terms corresponding to nondominant $\m'$-weights.  Then given $\la$, we should collect all the remaining terms in $\b(0)$ of the form $e^{\llbracket \la,*,*\rrbracket}$ (where the stars are arbitrary), and then factor out $e^{\llbracket \la,0,0\rrbracket}$ to obtain the multi-term ``coefficient'' of $e^\la$ indicated in \eqref{b0separate}.  The terms inside this coefficient tell us exactly which generalized Verma modules $M_{\mu,\nu}$ appear in the Enright resolution of $\widetilde F^\la_{p,q}$, along with appropriate signs depending on the parity of the term.

\begin{rem*}
In order to recover the complete data of the Enright resolution, we clearly need to supplement the method outlined above so as to determine the exact term (rather than just the parity) in which each generalized Verma module occurs.  Although not \textit{a priori} obvious, it can be seen from the construction in \cite{ew} that as we move from right to left in the resolution, the partitions $\mu$ and $\nu$ strictly increase in size.  This fact will allow us to easily recover the ordering of the terms once we have found the coefficient of $e^\la$.
\end{rem*}

\begin{exam}
\label{ex:n1p3q3}
We use software to present an example of the method outlined above.  On one hand, we will compute the terms of the Enright resolution of $\widetilde F^\la_{p,q}$ directly, using Maple code written by Jeb Willenbring.  On the other hand, we will expand $\b(0)$ and isolate the coefficient of $e^\la$, using Mathematica code written by the author of the present paper.

Let $n=1$, with $p=q=3$.  Set $\la = 0$, the empty partition; then $\widetilde F^0_{3,3} = \C[V]^{\GL_1}$ is the invariant subalgebra of $\C[V]$ in the setting of Section \ref{s:classicalinvariant}.  This example is of special interest because, by the second fundamental theorem of classical invariant theory, $\C[V]^{\GL_1}$ is isomorphic to the coordinate ring of the determinantal variety consisting of matrices in $\M_3(\C)$ with rank at most 1, known as the first Wallach representation of $\su(3,3)$.  See Enright and Hunziker's paper \cite{eh} on minimal resolutions for the Wallach representations; in fact, our present example is exactly Example 7.11 in \cite{ehp} where $p=3$.

In Maple, we compute the following resolution for $\widetilde F^0_{3,3}$, with $M_{\mu,\nu}$ as in \eqref{gvmform}:
\begin{align*}
0 \rightarrow M_{(2,2,2),(2,2,2)}\rightarrow M_{(2,1,1),(2,1,1)} &\rightarrow M_{(1,1,1),(2,1,0)} \oplus M_{(2,1,0),(1,1,1)} \\
&\rightarrow M_{(1,1,0),(1,1,0)}\rightarrow M_{(0,0,0),(0,0,0)}\rightarrow \widetilde F^0_{3,3} \rightarrow 0.
\end{align*}  
In Mathematica, we define the generating function $\b(0)$ directly from the definition \eqref{b0}, setting $x_i=e^{\varepsilon_i}$ for $i=1,2,3$, and $w=e^{\varepsilon_4}$, and $y_i = e^{\varepsilon_{i+4}}$ for $i=1,2,3$.  Note that the alphabetical order $w,x,y$ mirrors that of $n,p,q$ as we visualize the embedding $\m' \subset \g'$; in this way, the exponent vector for the variables $x_i$ encodes (the dual of) a weight of $\gl_p$, the exponent vector for the $y_i$ encodes a weight of $\gl_q$, and the exponent of $w$ encodes a weight of $\gl_n$.  Using \eqref{b0}, we obtain
$$
\b(0)=\frac{\displaystyle\prod_{1\leq i < j \leq 3}\left(1-\frac{x_j}{x_i}\right) \prod_{1\leq i < j \leq 3}\left(1-\frac{y_j}{y_i}\right) \prod_{1\leq i,j\leq 3}\left(1-\frac{y_j}{x_i}\right)}{\displaystyle\prod_{1\leq i \leq 3}\left(1-\frac{w}{x_i}\right) \prod_{1\leq i\leq 3}\left(1-\frac{y_i}{w}\right)}.
$$
Upon expanding $\b(0)$ to a sufficiently high order, we program Mathematica to retain only those terms in which the exponent vectors for $(x_1,x_2,x_3)$ and $(y_1,y_2,y_3)$ are both weakly decreasing, corresponding to dominant weights for $\gl_3$.  (Since $n=1$ in this example, there is no need to do the same for the lone variable $w$.)  In the remaining sum, we then find the ``coefficient'' of $e^\la$, namely, of $w^0$ --- that is, we collect all terms in which the power of $w$ is 0.  This ``coefficient'' is the following sum of terms in the variables $x_i$ and $y_i$; in light of the remark before this example, we write the terms in descending order with respect to total degree in the $y_i$:
\begin{align*}
    &\phantom{==} \frac{y_1^2 y_2^2 y_3^2}{x_1^2 x_2^2 x_3^2} - \frac{y_1^2 y_2 y_3}{x_1 x_2 x_3^2} + \frac{y_1^2 y_2}{x_1 x_2 x_3} + \frac{y_1 y_2 y_3}{x_2 x_3^2} - \frac{y_1 y_2}{x_2 x_3} + 1\\[3ex]
    &=\phantom{-}e^{(-2,-2,-2,0,2,2,2)} - e^{(-1,-1,-2,0,2,1,1)} + e^{(-1,-1,-1,0,2,1,0)}+e^{(0,-1,-2,0,1,1,1)}\\
    &\phantom{=}- e^{(0,-1,-1,0,1,1,0)}+e^{(0,0,0,0,0,0,0)}\\[3ex]
    &= \phantom{-}e^{\llbracket 0,(2,2,2),(2,2,2) \rrbracket} - e^{\llbracket 0,(2,1,1),(2,1,1) \rrbracket} + e^{\llbracket 0,(1,1,1),(2,1,0) \rrbracket} + e^{\llbracket 0,(2,1,0),(1,1,1) \rrbracket}\\
    &\phantom{=} - e^{\llbracket 0,(1,1,0),(1,1,0) \rrbracket} + e^{\llbracket 0,(0,0,0),(0,0,0) \rrbracket}.
\end{align*}
We thus arrive at the Enright resolution produced by Maple.
\end{exam}

As a final application of our main result, we can write down a new combinatorial expression, without any signs, for Blattner's formula in Type A, this time for \textit{two} noncompact roots (under a stability condition).

\begin{thm}
\label{thm:LRC2nc}
Suppose $n \geq p+q$.  Let $\g' = \sl_{n+p+q}$ with $\Pi'_{\nc}=\{\al_p,\al_{p+n}\}$; hence $\k' = \s(\gl_n \oplus \gl_{p+q})$, just as in Section \ref{s:2ncroots}. Let $\delta,\eta \in P_+(\k')$, where $\delta=\llbracket\delta^n,\delta^{p },\delta^q\rrbracket$ and $\eta = \llbracket \eta^n, \eta^{p}, \eta^q\rrbracket$, and write $\llbracket \delta^p,\delta^q\rrbracket:=(-\delta^p_p,\ldots,-\delta^p_1,\delta^q_1,\ldots,\delta^q_q) \in P_+(\gl_{p+q})$. Then
$$
B(\delta,\eta) = \sum_{\la,\mu,\nu} \LRC{\llbracket \delta^p,\delta^q\rrbracket}{\mu,\nu} \LRC{\eta^n}{\delta^n,\la} \LRC{\eta^p}{\mu,\la^+} \LRC{\eta^q}{\nu,\la^-}
$$
where the sum is over all $\la,\mu,\nu$ satisfying \eqref{howeconditions}.
\end{thm}

\begin{proof}
Using \eqref{b0final} to substitute for $\b(0)$, and then \eqref{epsilonstable} to simplify since we are in the stable range, we calculate that
\begin{align*}
    \b(\delta) &= \ch L_{\k'}(\delta) \cdot \b(0)\\
    &= \ch\left(F_n^{\delta^n} \otimes F_{p+q}^{\llbracket\delta^p,\delta^q\rrbracket}\right) \cdot \left(\Delta_{\m'} \cdot \sum_{\la,\al,\be} \varepsilon^\la_{\al,\be} \cdot \ch F_n^\la \otimes F^{\al}_p \otimes F^\be_q\right)\\
    &= \ch\left(F_n^{\delta^n} \otimes \bigoplus_{\mu,\nu} \LRC{\llbracket \delta^p,\delta^q\rrbracket}{\mu,\nu} F^\mu_p \otimes F^\nu_q \right) \cdot \left(\sum_\la \ch F_n^\la \otimes F^{\la^+}_p \otimes F^{\la^-}_q\right)\cdot \Delta_{\m'}\\
    &= \ch \left( \bigoplus_{\mu,\nu} \LRC{\llbracket \delta^p,\delta^q\rrbracket}{\mu,\nu} \left(F_n^{\delta^n} \otimes F^\mu_p \otimes F^\nu_q\right) \otimes \bigoplus_\la F^\la_n \otimes F^{\la^+}_p \otimes F^{\la^-}_q  \right) \cdot \Delta_{\m'} \\
    &=\Delta_{\m'} \cdot \ch\bigoplus_{\substack{\la,\phantom{x}\mu,\phantom{x}\nu,\\\eta^n,\eta^p,\eta^q}}\LRC{\llbracket \delta^p,\delta^q\rrbracket}{\mu,\nu} \LRC{\eta^n}{\delta^n,\la} \LRC{\eta^p }{\mu,\la^+} \LRC{\eta^q}{\nu,\la^-}\left(F^{\eta^n}_n \otimes F^{\eta^p}_p \otimes F^{\eta^q}_q\right)\\
    &= \sum_{\eta:=\llbracket \eta^n,\eta^p,\eta^q\rrbracket}\left(\sum_{\la,\mu,\nu} \LRC{\llbracket \delta^p,\delta^q\rrbracket}{\mu,\nu} \LRC{\eta^n}{\delta^n,\la} \LRC{\eta^p}{\mu,\la^+} \LRC{\eta^q}{\nu,\la^-} \right) e^\eta
    \end{align*}
    plus terms of the form $    \rule{0.5cm}{0.15mm}e^{\text{nondominant $\m'$-weight}}.$  Since $\m' \subset \k'$, nondominant $\m'$-weights are necessarily nondominant $\k'$-weights, and so we ignore all such nondominant terms as falling outside the hypothesis on $\eta$.  Hence by the definition of $\b(\delta)$, we conclude that
    $$
    B(\delta,\eta) = \sum_{\la,\mu,\nu} \LRC{\llbracket \delta^p,\delta^q\rrbracket}{\mu,\nu} \LRC{\eta^n}{\delta^n,\la} \LRC{\eta^p}{\mu,\la^+} \LRC{\eta^q}{\nu,\la^-}.
    $$
\end{proof}



\bibliographystyle{amsplain}
\bibliography{references}

\end{document}